\documentclass[oneside,english]{amsart}
\usepackage[T1]{fontenc}
\usepackage[latin9]{inputenc}
\usepackage{amsthm}
\usepackage{amstext}
\usepackage{amssymb}

\makeatletter
\numberwithin{equation}{section}
\numberwithin{figure}{section}
\theoremstyle{plain}
\newtheorem{thm}{Theorem}
  \theoremstyle{definition}
  \newtheorem{defn}[thm]{Definition}
  \theoremstyle{remark}
  \newtheorem*{rem*}{Remark}
  \theoremstyle{plain}
  \newtheorem{lem}[thm]{Lemma}
  \theoremstyle{definition}
  \newtheorem*{example*}{Example}
  \theoremstyle{plain}
  \newtheorem*{thm*}{Theorem}
  \theoremstyle{plain}
  \newtheorem{cor}[thm]{Corollary}

\makeatother

\usepackage{babel}

\begin{document}
\title[Attractor continuation of non-autonomously perturbed systems]{Local attractor continuation of non-autonomously perturbed systems}

\author{Martin Kell}

\date{March 17, 2011}

\address{\noindent \tiny Max-Planck-Institute for Mathematics in the Sciences,
Inselstr. 22-26, D-04103 Leipzig, Germany}

\email{mkell@mis.mpg.de}

\thanks{The author would like to thank the IMPRS {}``Mathematics in the
Sciences'' for financial support and his advisor, Prof. J\"urgen
Jost, and the MPI MiS for providing an inspiring research atmosphere.}

\subjclass[2000]{Primary: 37B55, 37B35, 37L15; Secondary: 37H99 }

\keywords{local attractor, non-autonomous perturbation, bounded noise}
\begin{abstract}
Using Conley theory we show that local attractors remain (past) attractors
under small non-autonomous perturbations. In particular, the attractors
of the perturbed systems will have positive invariant neighborhoods
and converge upper semicontinuously to the original attractor. 

The result is split into a finite-dimensional part (locally compact)
and an infinite-dimensional part (not necessarily locally compact).
The finite-dimen\-sional part will be applicable to bounded random
noise, i.e. continuous time random dynamical systems on a locally
compact metric space which are uniformly close the unperturbed deterministic
system. The {}``closeness'' will be defined via a (simpler version
of) convergence coming from singular perturbations theory.
\end{abstract}
\maketitle

\section{Introduction}

This paper uses methods from Conley index theory to establish a continuation
of isolated attractors. For these attractors there is a stable neighborhood,
i.e. a positive invariant isolating neighborhood. Traditionally the
Conley index is applied to isolated invariant set whose isolating
neighborhoods are bounded and satisfy some compactness assumption,
Rybakowski \cite{Rybakowski1987} calls them admissible neighborhoods.
Focusing on a stable neighborhood we show that a non-autonomously
perturbed system has also a stable neighborhood which is {}``close''
to the original one.

A similar result for parabolic PDEs was obtained by Prizzi in \cite{Prizzi2005}
for small almost-periodic perturbations with sufficiently high {}``frequency''.
Because almost-periodicity causes compactness of the {}``perturbation-space''
he can obtain non-empty invariant sets even for the unstable case
(for flows on locally compact spaces our result also applies to totally
unstable invariant sets). In \cite{Ward1994} Ward obtains the result
for ODEs and small perturbations satisfying some hypothesis (H1).
Our result generalizes both of them in a way that we don't need almost-periodicity
and applies to quite general semiflows even in the infinite-dimensional
setting.

An open question to us is if an unstable invariant set with non-trivial
Conley index disappears for all small non-autonomous perturbations.
The continuation result and the Wa{\.z}ewski principle only give
us a non-empty positive invariant set. Obviously this cannot happen
in dimension $1$ and the perturbation should neither be almost-periodic
nor satisfy Ward's hypothesis (H1).

Furthermore, our result implies the upper semicontinuity of the global
pullback attractor (see \cite{Caraballo2003}) if the perturbation
is {}``uniformly small''. In addition, this also holds for local
attractors (called past attractors in \cite{Rasmussen2007}).

The result (see section \ref{sec:Inf}) is a consequence of the translation
invariance of the unperturbed flow and standard continuation results
of the Conley index. The proof is essentially contained in {\cite[Theorem 12.3]{Rybakowski1987}}
after replacing all admissibility arguments by an appropriate version
(see also \cite{Carbinatto2002}).

For locally compact metric spaces $X$ we can use ideas from \cite{Benci1991}.
Because the proof is very clear and easy to understand, we are going
to show it in section \ref{sec:Fin} even though this case, which
we call finite-dimensional case, is contained in the infinite-dimensional
case. Another reason to give the proof is that the obtained stable
neighborhoods are flow-defined and thus applicable to random dynamical
system on locally compact metric spaces without further assumptions
(see section \ref{sub:rds}). This implies that sufficiently small
bounded noise does not destroy a local attractor. We don't need regularity
of the support of the noise used by Ruelle in \cite{Ruelle1981} and
thus we could generalize his result (see also {\cite[Appendix D]{Bonatti2005}}
and the reference therein).

Furthermore, using the ideas of \cite{Kell2011c} and adjusting the
definition of semi-singular admissibility the infinite-dimensional
case can be extended to the discrete time setting. Thus stable neighborhoods
of a local attractor of a discrete time dynamical system, i.e. a continuous
maps $f:X\to X$, can be continued under small non-autonomous perturbations
$\tilde{f}:\mathbb{Z}\times X\to\mathbb{Z}\times X$.

\subsection*{Preliminaries}
\begin{defn}
[local semiflow] Let $\pi:D\to X$ be a continuous map into a topological
space $X$ and $D$ open in $\mathbb{R}^{+}\times X$ with $\{0\}\times X\subset D$.
For all $x\in X$ we define \[
\omega_{x}=\sup\{t>0\,|\,(t,x)\in D\}\in(0,\infty].\]
Assume $D\cap\mathbb{R}\times\{x\}=[0,\omega_{x})\times\{x\}$. Then
$\pi$ is called a (local) semiflow if the following holds 
\begin{itemize}
\item $x\pi0=x$ for all $x\in X$
\item if $(t,x)\in D$ and $(s,x\pi t)\in D$ then $x\pi(s+t)$ is defined
and equals $(x\pi t)\pi s$
\end{itemize}
If, in addition, $\omega_{x}=\infty$ for all $x\in X$ then $\pi$
is called a global semiflow.\end{defn}
\begin{rem*}
We use the notation $x\pi t$ for $\pi(t,p)$ and $x\phi^{p}t=\phi(t,p,x)$
whenever $(t,p,x)\in D$ (see below). Furthermore, we write $x\phi^{p}[0,t_{0}]$
for the set $\{x\phi^{p}t\,|\, t\in[0,t_{0}]\}$ under the condition
that $t_{0}<\omega_{x}^{p}$, otherwise $x\phi^{p}[0,t_{0}]$ is not
defined.\end{rem*}
\begin{defn}
[non-autonomous dynamical system (NDS)] Let $X$ be a metric space
and $P$ be a set called the base set. A (local) NDS is a pair of
mappings\begin{eqnarray*}
\theta:\mathbb{R}\times P & \to & P\\
\phi:D & \to & X\end{eqnarray*}
such that the following holds:
\begin{itemize}
\item $\theta$ is a (not necessarily continuous) flow, called the base
flow
\item $D$ is open in $\mathbb{R}^{\ge0}\times P\times X$ with $\{0\}\times P\times X\subset D$,
$x\phi^{p}0=x$ and whenever $(s,p,x)\in D$ and $(s+t,p,x)\in D$
for some $s,t\ge0$\[
(t,p\theta s,x\phi^{p}s)\in D\]
and \[
x\phi^{p}(t+s)=(x\phi^{p}s)\phi^{p\theta s}t.\]
Furthermore, we define \[
\omega_{x}^{p}=\sup\{t>0\,|\,(t,p,x)\in D\ \}\]

\item $\phi$ is continuous with respect to $t\in\mathbb{R}_{\ge0}$ and
$x\in X$ for fixed $p\in P$.
\end{itemize}
We call $(\phi,\theta)$ a (local) semiprocess if $\theta$ is continuous
and $P$ a metric space.\end{defn}
\begin{rem*}
To every (semi)process $(\phi,\theta)$ we can associated a semiflow
$\pi:D\to P\times X$ in the following way. If $(t,p,x)\in D$ then
$(p,x)\pi t=(p\theta t,x\phi^{p}t)$. The resulting (semi)flow is
called skew product semiflow. 
\end{rem*}

\begin{rem*}
If $P$ is a compact metric space and $\theta$ continuous then standard
Conley index theory is applicable: The unperturbed system $\pi$ is
a product of the flow $\theta$ and a (semi)flow $\pi_{0}$ on $X$
so that a Conley index $h(N,\pi_{0})$ lifts to \[
h(P\times N,\theta\times\pi_{0})=h(P,\theta)\wedge h(N,\pi_{0}).\]
The stability result follows from a standard continuation result for
the index and $H_{0}(h(P\times N,\theta\times\pi_{0}))\ne0$ iff $H_{0}(h(N,\pi_{0}))\ne0$
(see \cite{Kell2010} for classification of stability via the zeroth
singular homology $H_{0}$ of the index).
\end{rem*}
Although most of the following can be formulated for general NDS if
we assume uniform convergence w.r.t. to the orbit set of $\theta$
(see section \ref{sub:non-auto}) we will restrict our attention to
$P=\mathbb{R}$ and $\tau_{t}(s):=\theta(t,s)=t+s$. This includes
processes generated by non-autonomous differential equations. Furthermore,
we will only look at the induced skew product (semi)flow, i.e. $\pi:D\subset\mathbb{R}^{\ge0}\times\tilde{X}\to\tilde{X}$
with $\tilde{X}=\mathbb{R}\times X$. 

A map $\sigma:J\to X$ with $J\subset\mathbb{R}$ is called a solution
of $\pi$ through $x_{0}\in X$ if $0\in J$ and $\sigma(0)=x_{0}$
and whenever $t,t+s\in J$ for some $s>0$ with $s<\omega_{\sigma(t)}$
then \[
\sigma(t+s)=\sigma(t)\pi s.\]

$\sigma$ is a left solution if $J\cap\mathbb{R}^{-}=(a,0]$ for some
$a\in[-\infty,0)$ and it is called a full left solution if $\mathbb{R}^{-}\subset J$. 

Let $Y\subset X$ be arbitrary. We define the following sets\begin{align*}
A^{+}(Y) & =\{x\in Y\,|\, x\pi t\in Y\,\text{for all }t\in[0,\omega_{x})\}\\
A^{-}(Y) & =\{x\in Y\,|\,\text{\ensuremath{\exists}a full left solution \ensuremath{\sigma\,}through \ensuremath{x\,}with \ensuremath{\sigma(\mathbb{R}^{-})\subset Y}}\}\end{align*}
and \[
A(Y)=A^{-}(Y)\cap A^{+}(Y).\]
$Y$ is called invariant if $Y=A(Y)$, positive invariant if $Y=A^{+}(Y)$
and negative invariant if $Y=A^{-}(Y)$. The sets $A(Y)$ and $A^{\pm}(Y)$
depend on $Y$ and $\pi$. In case we talk about several flows $\pi_{n}$
we will write $A_{\pi_{n}}(Y)$ and $A_{\pi_{n}}^{\pm}(Y)$.

If $S\subset\operatorname{int}N$ for some closed neighborhood $N$
and $S$ is the maximal invariant set in $N$, i.e. $S=A(N)$, then
$S$ is called an isolated invariant set (w.r.t. $\pi$). A closed
set $N$ is called isolating neighborhood if the maximal $\pi$-invariant
set is in the interior of $N$, i.e. $A(N)\subset\operatorname{int}N$.
In particular the closure $\operatorname{cl}U$ of a neighborhood
$U$ of $A(N)$ with $U\subset N$ is an isolating neighborhood.
\begin{defn}
A closed isolating neighborhood $N$ is called stable if \[
A(N)\subset\operatorname{int}N\]
 and $N$ is positive invariant, i.e. \[
A^{+}(N)=N.\]

\end{defn}

\section{\label{sec:Fin}Finite dimensional case}

Although many ideas in this section are taken from Conley index theory
we don't want to introduce the full theory of the Conley index. Many
of the techniques and definitions in this section are based on Benci's
paper \cite{Benci1991}. From now on let $X$ be locally compact.
Since our isolating neighborhood will be compact we will assume that
$\pi_{0}$ is a flow in order to simplify our arguments, i.e. there
will be no finite-time blow-up. This will be true for the skew product
flow induced by a process if the {}``$X$-component'' of set $\tilde{N}\subset\mathbb{R}\times X$
is compact as well (the unbounded component represents time).

From \cite{Benci1991} we take the following definitions:\[
G^{T}(N)=G_{\pi}^{T}(N)=\{x\in X\,|\, x\pi[-T,T]\subset\operatorname{cl}N\}\]
and\[
\Gamma^{T}(N)=\Gamma_{\pi}^{T}(N)=\{x\in G^{T}(N)\,|\, x\pi[0,T]\cap\partial N\ne\varnothing\}.\]
Furthermore we define the set of isolating neighborhoods as \[
\mathcal{F}=\mathcal{F}_{\pi}=\{N\subset X\,|\,\operatorname{int}N\ne\varnothing\,\mbox{and}\,\exists T>0\,\mbox{s.t.}\, G^{T}(N)\subset\operatorname{int}N\}.\]

\begin{lem}
[\cite{Benci1991}]If $N\in\mathcal{F}$ then the following hold:
\begin{enumerate}
\item $T_{1}>T_{2}>0$ then $G^{T_{1}}(N)\subset G^{T_{2}}(N)$
\item $G^{T}(N)$ and $\Gamma^{T}(N)$ are closed and $\Gamma^{T}(N)\subset\partial G^{T}(N)$
\item If $G^{T}(N)\subset\operatorname{int}N$ then $G^{2T}(N)\subset\operatorname{int}G^{T}(N)$
\item If $\Gamma^{T}(N)=\varnothing$ then $A^{+}(G^{T}(N))=G^{T}(N)$
\end{enumerate}
\end{lem}
\begin{rem*}
It can be shown that for large $T$ the pair $(G^{T}(N),\Gamma^{T}(N))$
defines an index pair (see section \ref{sec:Inf} for the definition).
We will focus only on stable neighborhoods for which $\Gamma^{T}(N)=\varnothing$
for large $T$. We can show that for arbitrary isolated invariant
set with non-trivial Conley index the index continues even for sufficiently
small non-autonomous perturbations. But the corresponding index pair
$(N_{1},N_{2})$ is unbounded. In particular, we cannot show in general
that $A(N_{1})\ne0$ if the index $(N_{1}/N_{2},[N_{2}])$ is non-trivial.
\end{rem*}
A standard result from Conley index theory is the following theorem.
This result also holds under the assumption that $N$ is strongly
$\pi$-admissible (see section \ref{sec:Inf} for the definition).
\begin{thm}
[{\cite[Corollary 5.5]{Rybakowski1987}}]\label{thm:stable} Let $N$
be a compact isolating neighborhood such that \[
A^{-}(N)=A(N)=K\ne\varnothing\]
then there exist a stable isolating neighborhood $B\subset\operatorname{int}N$
such that for all $x\in\partial B$ there is an $\epsilon>0$ such
that for all $t\in(0,\epsilon)$\[
x\pi(-t)\notin B\]
and \[
x\pi t\in\operatorname{int}B.\]
An isolating neighborhood with this property will be called stable
isolating block. \end{thm}
\begin{rem*}
Suppose we got $B'\subset N$ as a result from the theorem. Applying
the theorem again we get a second stable isolating block $B\subset B'$.
Because $B'$ is compact and $K$ is in its interior, $U_{\delta}(B)=\{x\in X\,|\, d(x,B)<\delta\}\subset B'$
for some $\delta>0$.
\end{rem*}
Let $\pi_{0}$ be a flow on $X$. Then we can define a skew product
flow $\pi$ on $\mathbb{R}\times X$ by \[
(s,x)\pi t=(\tau_{s}(t),x\pi_{0}t).\]
 Thus $\pi_{0}$ can be considered as a process which does not depend
on the base flow $\tau$ and is therefore translation invariant w.r.t.
time. 
\begin{defn}
[Semi-singular convergence]Let $\pi_{n}$ be a sequence of semiflows
on $\mathbb{R}\times X$ and $(s_{n},x_{n})_{n\in\mathbb{N}}$ be
any sequence in $\mathbb{R}\times X$ and $t_{n}\in\mathbb{R}^{\ge0}$
such that $x_{n}\to x_{0}$ and $t_{n}\to t_{0}$ then we say $\pi_{n}$
converges semi-singularly to $\pi_{0}$ if \[
P_{2}((s_{n},x_{n})\pi_{n}t_{n})\to x_{0}\pi_{0}t_{0}\]
where $P_{2}(s,x)=x$. We write $\pi_{n}\overset{\mbox{\tiny ssing}}{\longrightarrow}\pi_{0}$.
Semi-singular convergence implies that $\pi_{n}\to\pi$ in the usual
sense if each $\pi_{n}$ is skew product flow of processes because
if in addition $s_{n}\to s_{0}$ for some $s_{0}\in\mathbb{R}$ then
\[
(s_{n},x_{n})\pi_{n}t_{n}=(s_{n}+t_{n},P_{2}((s_{n},x_{n})\pi_{n}t_{n}))\to(s_{0}+t_{0},x_{0}\pi_{0}t_{0})=(s_{0},x_{0})\pi t_{0}.\]
\end{defn}
\begin{rem*}
Semi-singular convergence is a simplified version of singular convergence
defined in \cite{Carbinatto2002}. With the metric $d_{\epsilon}(t,s)=\epsilon\min\{|t-s|,1\}$
we recover the singular convergence.
\end{rem*}
Assume from now on that $\pi_{n}$ is a sequence of flows such that
$\pi_{n}\overset{\mbox{\mbox{\tiny ssing}}}{\longrightarrow}\pi_{0}$. 
\begin{example*}
The standard example is an autonomous ordinary differential equation
\[
\dot{x}=f_{0}(x)\]
and the non-autonomous ODEs\[
\dot{x}=f_{n}(t,x)\]
such that $\sup_{t}\|f_{n}(t,x)-f_{0}(x)\|<\epsilon_{n}$ with $\epsilon_{n}\to0$.
If $f_{0}$ and $f_{n}$ are locally Lipschitz continuous then they
induce a local flow $\pi_{0}$, resp. local processes $\phi_{n}$.
Restricted to some compact set in $X$ we can assume that $\pi_{0}$
and $\phi_{n}$ are defined everywhere. Suppose we have $(s_{n},x_{n})_{n\in\mathbb{N}}$
with $x_{n}\to x_{0}$ and $t_{n}\to t_{0}$. The ODEs defined by\[
\dot{x}=\tilde{f}_{n}(t,x)=f_{n}(s_{n}+t,x)\]
generate time-translated processes $\tilde{\phi}_{n}$. We still have
$\sup_{t}\|\tilde{f}_{n}(t,x)-f_{0}(x)\|<\epsilon_{n}$. The sequence
$(s_{n},x_{n})_{n\in\mathbb{N}}$ corresponds to $(0,x_{n})_{n\in\mathbb{N}}$.
Interpreting $\phi_{n}$ and $\tilde{\phi}_{n}$ as skew products
$\pi_{n}$ and $\tilde{\pi}_{n}$ we see that $\tilde{\pi}_{n}\to\pi=\tau\times\pi_{0}$
which implies semi-singular convergence of $\pi_{n}$.
\end{example*}
Suppose now $S\subset X$ is an attractor for $\pi_{0}$. Then there
exists a stable isolating block $B$. Thus $N:=\mathbb{R}\times B$
is a stable isolating block for $\pi$ with $A_{\pi}(N)=\mathbb{R}\times S$.
Furthermore, it can be shown that for all $T>0$ \begin{eqnarray*}
G_{\pi}^{T}(N) & = & \mathbb{R}\times G_{\pi_{0}}^{T}(B)\\
 & = & \mathbb{R}\times(B\pi_{0}T)\end{eqnarray*}
and $\Gamma_{\pi}^{T}(N)=\varnothing$. 

Since $B$ is compact $(\cdot)\pi t:N\to N$ is uniformly continuous
for all $t\ge0$. And for each $T>0$ there is some $\delta>0$ such
that $U_{\delta}(G_{\pi}^{T}(N))\subset N$ where $U_{\epsilon}(A)=\{x\in X\,|\, d(x,A)<\epsilon\}$.

We assume here that $\pi_{n}$ does not blow-up in finite time in
$N$. This will always be the case if $\pi_{n}$ is a skew product
flow because the unbounded component of $N=\mathbb{R}\times B$ represents
the time and $B$ is compact. Hence whenever $\omega_{(s,x)}^{\pi_{n}}<\infty$
for $(s,x)\in N$ then $(s,x)\pi_{n}[0,\omega_{(s,x)}^{\pi_{n}})\not\subset N$.
\begin{lem}
For every $T>0$ there is an $n_{0}>0$ such that for all $n\ge n_{0}$\[
G_{\pi_{n}}^{T}(N)\ne\varnothing.\]
\end{lem}
\begin{proof}
easy exercise\end{proof}
\begin{lem}
\label{lem:Gamma_empty}There is a $T_{0}$ and a $n_{0}$ such that
for all $T\ge T_{0}$ and $n\ge n_{0}$\[
\Gamma_{\pi_{n}}^{T}(N)=\varnothing.\]
\end{lem}
\begin{rem*}
The proof is based on {\cite[Lemma 3.5]{Carbinatto2002}}. The idea
is to show that if it does not hold then $K=A_{\pi_{0}}(B)$ and $\partial B$
intersect non-trivial which is impossible because $K$ is isolated.
We will give the whole proof because the fact that $\tilde{z}\in A_{\pi_{0}}^{-}(B)$
will also apply to the infinite-dimensional case if we assume $\{\pi_{n}\}$-semi-singular-admissibility
of $N$. \end{rem*}
\begin{proof}
Suppose this is not the case. Then there exist a sequence $T_{n}\to\infty$
and a subsequence of $\pi_{n}$, also denoted by $\pi_{n}$, such
that $E_{n}:=\Gamma_{\pi_{n}}^{T_{n}}(N)\ne\varnothing$ for all $n$.

This implies that there is a sequence $x_{n}\in N$ such that for
some $s_{n}\in[0,T_{n}]$ we have $x_{n}\pi_{n}[0,t_{n}]\subset N$,
$x_{n}\pi_{n}T_{n}\in E_{n}$ and $x_{n}\pi t_{n}\in\partial N$ with
$t_{n}=T_{n}+s_{n}$.

Since $N=\mathbb{R}\times B$ we have $\partial N=\mathbb{R}\times\partial B$.
So that $x_{n}\pi t_{n}=(r_{n},\tilde{z}_{n})$ for some $\tilde{z}_{n}\in\partial B$.
Because $\partial B$ is compact there is a subsequence of $x_{n}\pi_{n}t_{n}$
denoted by $z_{n}^{0}=x_{n}^{0}\pi_{n}^{0}t_{n}^{0}$ such that $\tilde{z}_{n}^{0}\to\tilde{z}^{0}\in\partial B$. 

Since $B$ is compact, $t_{n}^{0}\to\infty$ and $x_{n}^{0}\pi_{n}^{0}[0,t_{n}^{0}]\subset N$
there is a subsequence $(x_{n}^{1}\pi_{n}^{1}t_{n}^{1})_{n\in\mathbb{N}}$
with $t_{n}^{1}\ge1$ such that $(r_{n}^{1},\tilde{z}_{n}^{1})=x_{n}^{1}\pi_{n}^{1}(t_{n}^{1}-1)$
is defined and $\tilde{z}_{n}^{1}\to\tilde{z}^{1}\in B$. Recursively,
for each $k\ge1$ we get a subsequence $(x_{n}^{k}\pi_{n}^{k}t_{n}^{k})_{n\in\mathbb{N}}$
of $(x_{n}^{k-1}\pi_{n}^{k-1}t_{n}^{k-1})_{n\in\mathbb{N}}$ with
$t_{n}^{k}\ge k$ such that $z_{n}^{k}=(r_{n}^{k},\tilde{z}_{n}^{k})=x_{n}^{k}\pi_{n}^{k}(t_{n}^{k}-k)$
and $\tilde{z}_{n}^{k}\to\tilde{z}^{k}$ for some $\tilde{z}^{k}\in B$.

We claim that $\tilde{z}^{k}\pi_{0}t\in N$ all $t\in[0,k]$ and $\tilde{z}^{k}\pi_{0}k=\tilde{z}^{0}$,
i.e. $\tilde{z}^{0}\in A_{\pi_{0}}^{-}(B)$. Postponing the proof
of this claim we immediately get that $\tilde{z}^{0}\in\partial B\cap A^{-}(B)$
but $B$ is a stable isolating neighborhood such that $A(B)=A^{-}(B)\subset\operatorname{int}B$,
i.e. $\partial B\cap A_{\pi_{0}}^{-}(B)=\varnothing$, which is a
contradiction and thus the lemma is true.

It remains to proof our claim. Suppose $\tilde{z}^{k}\pi_{0}t_{0}\notin B$
for some $t_{0}\in[0,k]$. Because $\pi_{n}\overset{\mbox{\tiny ssing}}{\longrightarrow}\pi_{0}$
the sequence $\tilde{y}_{n}$ with $(s_{n},\tilde{y}_{n})=z_{n}^{k}\pi_{n}^{k}t_{0}$
converges to $\tilde{z}^{k}\pi_{0}t_{0}$, but $\tilde{y}_{n}\in B$
and $B$ is closed which is impossible. Hence $\tilde{z}^{k}\pi_{0}[0,k]\subset B$. 
\end{proof}
Combining the two previous lemmas we obtain that there is a $T_{0}$
and an $n_{0}$ such that for all $T\ge T_{0}$ and $n\ge n_{0}$
the set $G_{\pi_{n}}^{T}(N)\ne\varnothing$ is a stable isolating
neighborhood for $\pi_{n}$.

Because $B$ is compact and $(\cdot)\pi t$ is uniformly continuous
on $N$, there is a $\delta=\delta(\pi_{0},B)>0$ such that for $T\ge T_{0}$
\begin{eqnarray*}
U_{\delta}(G_{\pi}^{2T}(N)) & \subset & G_{\pi}^{T}(N)\end{eqnarray*}
and \begin{eqnarray*}
U_{\delta}(G_{\pi}^{T}(N)) & \subset & N.\end{eqnarray*}

Furthermore, choose $\delta$ small enough and $T$ large enough so
that $U_{\delta}(B)\subset B'$ (see remark after theorem \ref{thm:stable})
and \[
G_{\pi_{0}}^{T}(B')\subset B,\]
i.e. $U_{\delta}(G_{\pi}^{T}(U_{\delta}(N)))\subset N$ holds as well.
\begin{rem*}
Benci (\cite{Benci1991}) defines a set $\Sigma_{0}$ which contains
all sets having a similar property and uses it to prove the continuation
of the Conley index. Most proofs of the inclusions like $U_{\delta}(G_{\pi}^{2T}(N))\subset G_{\pi_{n}}^{T}(N)$
are omitted in that paper. 
\end{rem*}
In addition assume now that each $\pi_{n}$ is a skew product flow,
i.e. \[
P_{1}((s,x)\pi_{n}t)=s+t.\]

\begin{thm}
\label{thm:finite}If for large $n\ge n_{0}$ and $T\ge T_{0}$ we
have \[
d(x\pi_{n}t,x\pi t)<\frac{\delta}{3}\]
 for all $t\in[-T,T]$ and $x\in N$ then\[
A_{\pi_{n}}(N)\ne\varnothing\]
is an isolated invariant set with compact $t$-slices such that for
an $\epsilon=\epsilon(\pi_{0},B,T)>0$ \[
U_{\epsilon}(A_{\pi_{n}}(N))\subset G_{\pi_{n}}^{T}(N).\]
\end{thm}
\begin{rem*}
The theorem shows that $A_{\pi_{n}}(N)$ is a past attractor in the
sense of Rasmussen \cite{Rasmussen2007}.\end{rem*}
\begin{proof}
Let $x\in G_{\pi_{n}}^{T}(N)$ then \[
x\pi_{n}[-T,T]\subset N.\]
 Since $d(x\pi t,x\pi_{n}t)<\delta$ for all $t\in[-T,T]$ we have
\[
x\pi[-T,T]\subset U_{\delta}(N).\]
This implies that $G_{\pi_{n}}^{T}(N)\subset G_{\pi}^{T}(U_{\delta}(N))$
and thus $U_{\delta}(G_{\pi_{n}}^{T}(N))\subset N$ because $U_{\delta}(G_{\pi}^{T}(U_{\delta}(N)))\subset N$

Because $\pi$ is uniformly continuous on $N\times[-T,T]$ we can
choose $\epsilon>0$ such that $d(x,\tilde{x})<\epsilon$ implies
$d(x\pi t,\tilde{x}\pi t)<\frac{\delta}{3}$ for all $x,\tilde{x}\in N$
and $t\in[-T,T]$. Now if $x\in U_{\epsilon}(G_{\pi_{n}}^{2T}(N))$
then there is an $\tilde{x}\in G_{\pi_{n}}^{2T}(N)$ such that $d(x,\tilde{x})<\epsilon$
and \[
\tilde{x}\pi_{n}[-T,T]\subset G_{\pi_{n}}^{T}(N).\]
Furthermore, we have for all $t\in[-T,T]$ \begin{eqnarray*}
d(x\pi_{n}t,\tilde{x}\pi_{n}t) & \le & d(x\pi_{n}t,x\pi t)+d(x\pi t,\tilde{x}\pi t)+d(\tilde{x}\pi t,\tilde{x}\pi_{n}t)\\
 & < & \frac{\delta}{3}+\frac{\delta}{3}+\frac{\delta}{3}=\delta.\end{eqnarray*}
Since $U_{\delta}(G_{\pi_{n}}^{T}(N))\subset N$ we have $x\pi_{n}[-T,T]\subset N$,
i.e. $U_{\epsilon}(G_{\pi_{n}}^{2T}(N))\subset G_{\pi_{n}}^{T}(N)$.
In particular, $U_{\epsilon}(A_{\pi_{n}}(N))\subset G_{\pi_{n}}^{T}(N)$.

Now let $x\in G_{\pi}^{2T}(N)$ then $x\pi[-T,T]\subset G_{\pi}^{T}(N)$.
Because $U_{\delta}(G_{\pi_{n}}^{T}(N))\subset N$ and $d(x\pi t,x\pi_{n}t)<\delta$
for all $t\in[-T,T]$ we have $x\in G_{\pi_{n}}^{T}(N)$. Hence \[
G_{\pi}^{2T}(N)=\mathbb{R}\times G^{2T}(B)\subset G_{\pi_{n}}^{T}(N).\]

Since $G_{n}:=G_{\pi_{n}}^{T}(N)$ is closed the set $G_{n}(t)=G_{\pi_{n}}^{T}(N)\cap\{t\}\times X\subset\{t\}\times B$
is compact and non-empty for each $t\in\mathbb{R}$. Now define\[
A(t)=\bigcap_{t\ge0}G_{n}(-t)\pi_{n}t\subset\{t\}\times B.\]
$G_{n}$ is positive invariant so that this intersection is a decreasing
sequence of non-empty compact sets and thus compact and non-empty
itself. Let \[
A=\bigcup_{t\in\mathbb{R}}A(t).\]
Then $A$ is invariant w.r.t. $N$ and $A(N)\subset A\subset N$ which
implies $A(N)=A\ne\varnothing$.
\end{proof}

\subsection{\label{sub:non-auto}Non-autonomous dynamical systems}

For a general NDS $(\phi,\theta)$ we do not have $P=\mathbb{R}$.
Assume $X$ is locally compact and $(\phi,\theta)$ is a two-sided
process. Denote by $P^{*}$ the set of $\theta$-orbits of $P$. Then
each orbit $\sigma\in P^{*}$ is represented by a $p\in P$. So after
choosing one representative for each orbit we can assume w.l.o.g.
$P^{*}\subset P$. 

Because $(\phi,\theta)$ is two-sided for each $p\in P^{*}$ we have
an invertible skew product flow $\pi_{p}$ with \[
(s,x)\pi_{p}t=(s+t,x\phi^{p\theta s}t).\]
Suppose $B$ is a stable isolating neighborhood of $\pi_{0}$. If
each $\pi_{p}$ is close to $\pi=\tau\times\pi_{0}$, i.e. there is
a sufficiently small $\delta>0$ and a $T>0$ (both independent of
$p$) such that \[
d((s,x)\pi t,(s,x)\pi_{p}t)<\delta\]
for all $t\in[-T,T]$ and $x\in B$ then theorem \ref{thm:finite}
implies \[
G_{p}=G^{T}(\mathbb{R}\times B)\]
is a stable isolating neighborhood for $\pi_{p}$ and their isolated
invariant sets $A_{p}=A_{\pi_{p}}(\mathbb{R}\times B)$ are non-empty.
Now define the set-valued map $D:P\to2^{X}$ by \[
D(q)=\{x\in X\,|\,\mbox{for some \ensuremath{s\in\mathbb{R}}, \ensuremath{p\in P^{*}}\,\ s.t. \ensuremath{(s,x)\in G_{p}}\,\ and }p\theta s=q\}\]
It is an easy exercise to show that $D$ maps into the set of closed
sets and is forward invariant w.r.t. the non-autonomous system $(\phi,\theta)$.
In the same way we can show that there is set-valued map $A$ which
is invariant and \[
A(q)=P_{2}((A_{p}\pi_{p}s)\cap\{s\}\times X)\]
is compact for $q=p\theta s$ and $p\in P^{*}$. Furthermore, the
$\epsilon$ in theorem \ref{thm:finite} only depends on $T$, $B$
and $\pi_{0}$, i.e. \[
U_{\epsilon}(A(q))\subset D(q)\]
which shows that $A$ is a past attractor.

\subsection{\label{sub:rds}Random dynamical systems}

A random dynamical system is an NDS such that $P=\Omega$ is a probability
space, $\theta$ is measurable dynamical system and the map\[
(t,\omega)\mapsto x\phi^{\omega}t\]
is measurable for every $x$ in $X$.

If we show that the sets $D$ and $A$ are random closed sets then
$A$ is a random past attractor. And we get:
\begin{thm*}
Sufficiently small (bounded) noise does not destroy local attractors.
\end{thm*}
In particular, if the global attractor of the unperturbed system has
several disjoint local attractors ({}``sinks'') then the same holds
for the global attractor of the perturbed system. This is a complementary
result to Crauel, Flandoli - {}``Additive Noise Destroys a Pitchfork
Bifurcation'' \cite{Crauel1998} where they show that the {}``perturbed''
global attractor is a single point with probability $1$ when {}``small''
white noise is applied.

It remains to show that $D$ and $A$ are closed random sets. It suffices
to show that $D$ is random compact. Each $G_{\omega}^{T}(N)$ is
flow-defined, i.e.\[
G_{\omega}^{T}(N)=\bigcap_{t\in[-T,T]}(\mathbb{R}\times B)\pi_{\omega}t.\]
Because the $\pi_{\omega}$ are skew products, for each $s$-slice
we have \begin{eqnarray*}
G_{\omega}^{T}(N)(s) & = & G_{\omega}^{T}(N)\cap(\{s\}\times X)\\
 & = & \{s\}\times\bigcap_{t\in[-T,T]}(B\phi^{\omega\theta(s-t)}t).\end{eqnarray*}
So that for $\varpi=\omega\theta s$ \begin{eqnarray*}
D(\varpi) & = & \bigcap_{t\in[-T,T]}(B\phi^{\varpi\theta(-t)}t)\\
 & = & \bigcap_{t\in[-T,T]\cap\mathbb{Q}}(B\phi^{\varpi\theta(-t)}t)\end{eqnarray*}
Because $B$ is a deterministic compact set, $\phi$ invertible and
$D$ a decreasing intersection of compact sets, $D$ is a random compact
set. In particular, we see that the choice of the representative $\omega\in\Omega^{*}=P^{*}$
does not matter and each $D(\varpi)$ is defined in the same way.

\section{\label{sec:Inf}Infinite dimensional case}

In this section we show that the index pair continuation in {\cite[Theorem 12.3]{Rybakowski1987}
still holds if we replace the admissibility arguments by semi-admissibility
and semi-singular convergence arguments. We are not able to show that
two index pairs of the perturbed (non-autonomous) system are Conley
equivalent, i.e. the quotients have the same homotopy type in the
category of pointed space. So the Conley index is not well-defined.
Even worse, because the index pair is unbounded the Wa{\.z}ewski
principle does not imply that the invariant set is non-empty if the
quotient space is non-trivial. Nevertheless, we can use the continuation
to show that a stable isolating neighborhood continues to a stable
isolating neighborhood, i.e. the exit set of the perturbed system
is empty.
\begin{rem*}
As mentioned in the introduction appropriate versions of theorem \ref{thm:cont}
and corollary \ref{cor:cont} hold also for stable isolating neighborhoods
in the discrete time setting using the ideas of \cite{Kell2011c}.
The definitions of semi-(singular-)admissibility and semi-singular
convergence can be given similarly. 

Furthermore, for both, continuous and discrete time, instead of a
semi-admis\-sibility of the perturbed maps we could use {}``weak
properness'' of the space $X$, i.e. bounded $\delta$-neighborhoods
of compact sets are weakly compact, to show that the invariant set
is non-empty and a weak attractor (see \cite{Kell2011c} for definition).\end{rem*}
\begin{defn}
[Semi-singular admissiblity] Let $\pi_{n}$ be semiflows on $\mathbb{R}\times X$
and $\pi_{0}$ be a semiflow on $X$ such that $\pi_{n}\overset{\mbox{\mbox{\tiny ssing}}}{\longrightarrow}\pi_{0}$.
A closed subset $N$ of $\mathbb{R}\times X$ is called $\{\pi_{n}\}$-semi-singular-admissible
(ss-admissible) if the following holds:
\begin{itemize}
\item Let $(s_{n})_{n\in\mathbb{N}}$ be a sequence in $\mathbb{R}$. If
$x_{n}\in X$ and $t_{n}\in\mathbb{R}^{+}$ with $0\le t_{n}<\omega_{(s_{n},x_{n})}^{\pi_{n}}$
are two sequences such that $t_{n}\to\infty$ as $n\to\infty$ and
$(s_{n},x_{n})\pi_{n}[0,t_{n}]\subset N$ then the sequence $\{P_{2}((s_{n},x_{n})\pi_{n}t_{n})\}_{n\in\mathbb{N}}$
of the projected endpoints has a convergent subsequence.
\end{itemize}
If, in addition, each $\pi_{n}$ does not explode in $N$, i.e. \[
(s,x)\pi_{n}[0,\omega_{(s,x)}^{\pi_{n}})\subset N\:\mbox{ implies }\:\omega_{(s,x)}^{\pi_{n}}=\infty\]
then we say that $N$ is strongly $\{\pi_{n}\}$-ss-admissible.
\end{defn}
The definition means that $\pi_{n}$ should not converge too badly
to $\pi_{0}$. It is a property of the sequence $\{\pi_{n}\}_{n\in\mathbb{N}}$
and does not tell anything about $\pi_{0}$ alone and neither about
a single $\pi_{n}$. An adjusted admissibility definition of a single
semiflow $\pi$ for our setting is given as follows.

\begin{defn}
[Semi-Admissiblity] Let $\pi$ be a semiflow on $\mathbb{R}\times X$.
A closed subset $N$ of $\mathbb{R}\times X$ is called $\pi$-semi-admissible
if the following holds:
\begin{itemize}
\item Let $(s_{n})_{n\in\mathbb{N}}$ be a sequence in $\mathbb{R}$. If
$x_{n}\in X$ and $t_{n}\in\mathbb{R}^{+}$ with $0\le t_{n}<\omega_{(s_{n},x_{n})}$
are two sequences such that $t_{n}\to\infty$ as $n\to\infty$ and
$\{s_{n}+t_{n}\}_{n\in\mathbb{N}}$ is precompact, $(s_{n},x_{n})\pi[0,t_{n}]\subset N$
then the sequence of endpoints $\{(s_{n},x_{n})\pi t_{n}\}_{n\in\mathbb{N}}$
has a convergent subsequence.
\end{itemize}
If, in addition, $\pi$ does not explode in $N$, i.e. \[
(s,x)\pi[0,\omega_{(s,x)})\subset N\:\mbox{ implies }\:\omega_{(s,x)}=\infty\]
then we say that $N$ is strongly $\pi$-semi-admissible.\end{defn}
\begin{rem*}
For a semiflow $\pi_{0}$ on $X$ the usual strong $\pi_{0}$-admissibility
of a closed subset $N$ is defined without the sequence $s_{n}$ (see
{\cite[4.1]{Rybakowski1987}}).
\end{rem*}
Suppose $B$ is closed and strongly $\pi_{0}$-admissible, $B\subset U$
for some open $U\subset X$ with strong $\pi_{0}$-admissible closure
and $d(B,\partial U)\ge\delta$. We say that $\pi_{n}\overset{\mbox{\mbox{\tiny ssing}}}{\longrightarrow}\pi_{0}$
converges uniformly if $d(P_{2}((s,x))\pi_{n}t^{*},x\pi_{0}t^{*})<\epsilon_{n,t}\to0$
for $t>0$ and all $(s,x)\in\mathbb{R}\times U$ with $t<\min\{\omega_{x}^{\pi_{0}},\omega_{(s,x)}^{\pi_{n}}\}$. 
\begin{rem*}
Uniform convergence of $\pi_{n}\overset{\mbox{\mbox{\tiny ssing}}}{\longrightarrow}\pi_{0}$
is similar to the assumption of Benci \cite{Benci1991} used to prove
his continuation theorem for the Conley index. \end{rem*}
\begin{lem}
Suppose $\pi_{n}\overset{\mbox{\mbox{\tiny ssing}}}{\longrightarrow}\pi_{0}$
uniformly and $\pi_{n}$ does not explode in $N=\mathbb{R}\times B$.
Then $N$ is strongly $\{\pi_{n}\}$-admissible.\end{lem}
\begin{proof}
Let $(s_{n},x_{n})\in N$ and $t_{n}\to\infty$ be a sequence fulfilling
the assumption of the definition of ss-admissibility. Because of uniform
convergence and $d(N,\partial U)\ge\delta$, there is an $N(t)$ for
each such that for all $n\ge N(t)$\[
x\pi_{0}[0,t]\subset U\]
whenever $(s,x)\in N$ with $(s,x)\pi_{n}t\subset N$.

Thus there is a sequence $r_{n}\ge0$ with $t_{n}-r_{n}\to\infty$
such that $(s_{n},x_{n})\pi_{n}[0,t_{n}]\subset N$ implies that \[
y_{n}\pi_{0}[0,t_{n}-r_{n}]\subset U\]
for $(\tilde{s}_{n},y_{n})=(s_{n},x_{n})\pi_{n}r_{n}$. Furthermore,
we can choose $r_{n}$ such that $\delta_{n}=\max_{t\in[0,t_{n}-r_{n}]}\epsilon_{n,t}\to0$
and therefore \[
d(P_{2}((s_{n},x_{n})\pi_{n}t_{n})),y_{n}\pi(t_{n}-r_{n}))\le\delta_{n}.\]

Because the closure of $U$ is strongly $\pi_{0}$-admissible, the
sequence of endpoints $\{y_{n}\pi(t_{n}-r_{n})\}_{n\in\mathbb{N}}$
has a convergent subsequence which implies that $\{P_{2}((s_{n},x_{n})\pi_{n})\}_{n\in\mathbb{N}}$
has a convergent subsequence, in particular the limit point is in
$N$.\end{proof}
\begin{defn}
[index pair]Let $\tilde{N}$ be a closed isolating neighborhood for
$K=A(\tilde{N})$. A pair $(N,L)$ is called an index pair in $\tilde{N}$
(w.r.t. $\pi$) if $L\subset N\subset\tilde{N}$ and the following
holds
\begin{enumerate}
\item $K\subset\operatorname{int}(N\backslash L)$ and $K$ is the maximal
invariant set in $\operatorname{cl}(N\backslash L)$
\item if $x\in L$ and $x\pi[0,\epsilon]\subset N$ then $x\pi[0,\epsilon]\subset L$,
i.e. $L$ is positive $\pi$-invariant relative to $N$
\item if $x\in N$ and $x\pi[0,\omega_{x})\not\subset N$ then there is
a $0\le t<\omega_{x}$ such that $x\pi t\in L$, i.e. $L$ is an exit
ramp for $N$
\end{enumerate}
\end{defn}
The following theorem is a variant of {\cite[I-12.3-12.7]{Rybakowski1987}}
(see also \cite{Carbinatto2002}) by replacing all admissibility arguments
with semi-singular-admissibility argument and replacing $g^{-}:\tilde{N}\to\mathbb{R}^{\ge0}$
by a lifted version $\tilde{g}^{-}:\mathbb{R}\times\tilde{N}\to\mathbb{R}^{\ge0}$
defined by $\tilde{g}^{-}(s,x)=g^{-}(x)$.
\begin{thm}
\label{thm:cont}Let $\pi_{0}$ be a local semiflow on $X$ and $\pi_{n}$,
$n\in\mathbb{N}$, be local semiflows on $\mathbb{R}\times X$. Suppose
$\tilde{N}$ is a closed set in $X$ and strongly $\pi_{0}$-admissible.
Moreover, assume $\pi_{n}\overset{\mbox{\mbox{\tiny ssing}}}{\longrightarrow}\pi_{0}$
as $n\to\infty$ and $\mathbb{R}\times\tilde{N}$ is strongly $\{\pi_{n_{m}}\}$-semi-admissible
for all subsequences $\{\pi_{n_{m}}\}_{m\in\mathbb{N}}$. Set $K_{n}=A_{\pi_{n}}(\mathbb{R}\times\tilde{N})$,
$K_{0}=A_{\pi}(\mathbb{R}\times\tilde{N})$ and assume $K=A_{\pi_{0}}(\tilde{N})\subset\operatorname{int}\tilde{N}$.
Then there exist an $n_{0}\in\mathbb{N}$ and two closed set $N$
and $N^{'}$ such that for each $n\in\{n\ge n_{0}\}\cup\{0\}$ there
are two index pair $\langle N_{n},L_{n}\rangle$ and $\langle N_{n}^{'},L_{n}^{'}\rangle$
with the following properties:
\begin{itemize}
\item $K_{n}\subset N^{'}\subset\operatorname{int}(\mathbb{R}\times N)\subset\operatorname{int}(\mathbb{R}\times\tilde{N})$,
$N_{0}^{'}=\mathbb{R}\times N^{'}$ and $N_{0}=\mathbb{R}\times N$
\item $\langle N_{n}^{'},L_{n}^{'}\rangle$ (resp. \textup{$\langle N_{n},L_{n}\rangle$)
is an index pair in $N_{0}^{'}$ (resp. in $N_{0}$) w.r.t $\pi_{n}$
(resp. w.r.t $\pi=\tau\times\pi_{0}$ if $n=0$)}
\item $N_{n}^{'}\subset N_{0}^{'}\subset N_{n}\subset N_{0}$ and $L_{n}^{'}\subset L_{0}^{'}\subset L_{n}\subset L_{0}$.
\end{itemize}
\end{thm}
\begin{proof}
Since the proof will follow the original one almost completely and
we are only interested in attractors, we only show a proof in case
$\tilde{N}$ is positive $\pi_{0}$-invariant. In particular, we will
assume that $\tilde{N}$ is an isolating block with empty exit set
defined via the function $g^{-}$. For $a>0$ define \[
V(a)=\{x\in\tilde{U}\,|\, g^{-}(x)<a\}.\]

Then there is an $a_{0}>0$ such that $N:=\operatorname{cl}V(a_{0})\subset\tilde{U}$.
And similar to {\cite[I-4.5]{Rybakowski1987}} we can show that for
$0<\epsilon\le a_{0}$ and all $n\ge n_{0}(\epsilon)$ \[
K_{n}\subset\mathbb{R}\times V(\epsilon).\]

Define \[
N_{n}(\epsilon)=(\mathbb{R}\times N)\cap\operatorname{cl}\{\tilde{y}\,|\,\mbox{\ensuremath{\exists\tilde{x}\in\mathbb{R}\times V(\epsilon)}, \ensuremath{t\ge0}\,\ s.t.\,\ \ensuremath{\tilde{x}\pi_{n}[0,t]\subset\mathbb{R}\times\tilde{U}\,}and \ensuremath{\tilde{x}\pi_{n}t=\tilde{y}}}\}.\]
Following the proof of {\cite[I-12.5]{Rybakowski1987}} we can show
that $N_{n}(\epsilon)$ satisfies the following properties for $n\ge n_{0}(\epsilon)$
\begin{itemize}
\item $x\in N_{n}(\epsilon)$ and $x\pi_{n}[0,t]\subset\mathbb{R}\times N$
implies $x\pi_{n}t\in N_{n}(\epsilon)$
\item $K_{n}\subset\mathbb{R}\times V(\epsilon)\subset N_{n}(\epsilon)$
\end{itemize}
We claim that for small $\epsilon_{0}>0$ whenever $\epsilon\le\epsilon_{0}$
and $n\ge n_{0}(\epsilon)$ then $N_{n}(\epsilon)$ is positive $\pi_{n}$-invariant.
If this is not true then there is a sequence $\epsilon_{m}\to0$ and
\[
y_{m}=(s_{m},z_{m})\in N_{n_{m}}(\epsilon_{m})\cap(\mathbb{R}\times\partial N).\]
 By definition of $N_{n_{m}}(\epsilon_{m})$ there is a sequence $\tilde{y}_{m}\in X$,
$x_{m}\in\mathbb{R}\times V(\epsilon_{m})$ and $t_{m}\ge0$ such
that $d(y_{m},\tilde{y}_{m})<2^{-m}$, $x_{m}\pi_{n_{m}}[0,t_{m}]\subset\tilde{U}$
and $\tilde{y}_{m}=x_{m}\pi_{n_{m}}t_{m}$. Because $g^{-}(x_{m})\to0$
and $A_{f}^{-}(B)=A_{f}(B)$ we can assume w.l.o.g. that $x_{m}\to x_{0}\in A_{f}(B)$.
Admissibility and $\pi_{n_{m}}\overset{\mbox{\mbox{\tiny ssing}}}{\longrightarrow}\pi_{0}$
imply the sequence $\{P_{2}(x_{m}\pi_{n_{m}}t_{m})\}_{m\in\mathbb{N}}$
has a convergent subsequence and w.l.o.g. $P_{2}(\tilde{y}_{m})=P_{2}(x_{m}\pi_{n_{m}}t_{m})\to z_{0}\in A_{f}^{-}(\tilde{N})=A_{f}(\tilde{N})\subset\operatorname{int}N$
(see proof of lemma \ref{lem:Gamma_empty}) and thus $P_{2}(y_{m})\to z_{0}$.
Since $P_{2}(z_{m})\in\partial N$ we must have $y_{0}\in\partial N$,
but this is a contradiction since $\partial N\cap A_{\pi_{0}}(\tilde{N})=\varnothing$.

Set $N^{'}=\operatorname{cl}V(\epsilon)$ which is also positive $\pi_{0}$-invariant.
Applying the arguments above for $N'$ we get another positive $\pi_{n}$-invariant
$N_{n}^{'}\subset\mathbb{R}\times N^{'}$. So we have \[
N_{n}^{'}\subset\mathbb{R}\times N^{'}\subset N_{n}\subset\mathbb{R}\times N.\]
 Positive invariance and $K_{n}\subset\mathbb{R}\times V(\delta)\subset N_{n}^{'}\subset N_{n}$
for some $\delta>0$ implies that $N_{n}$ and $N_{n}^{'}$ are index
pairs.\end{proof}
\begin{cor}
\label{cor:cont}If, in addition to the assumption of the theorem,
we assume that each $\pi_{n}$ is skew product flows and $\tilde{N}$
is a stable isolating neighborhood for $\pi_{0}$ and $N_{0}=\mathbb{R}\times N$
from the theorem is strongly $\pi_{n}$-semi-admissible then $K_{n}\ne\varnothing$
and each $t$-slice $K_{n}\cap\{t\}\times X$ is non-empty and compact.
Furthermore, there is an $n_{0}=n_{0}(\tilde{N},\pi_{0})$ and a $\delta=\delta(\tilde{N},\pi_{0})>0$
such that for all $n\ge n_{0}$ \[
U_{\delta}(K_{n})\subset N_{n},\]
i.e. $K_{n}$ is a past attractor in the sense of Rasmussen \cite{Rasmussen2007}.
And for the sequence $(K_{n})_{n\in\mathbb{N}}$ we have \[
\lim_{n\to\infty}\sup_{y\in K_{n}}\inf_{x\in K}d(P_{2}y,x)=0,\]
i.e. $K_{n}$ is upper-semicontinuous {}``at $K_{0}$'' as $n\to\infty$.\end{cor}
\begin{rem*}
In \cite{Caraballo2003} an upper-semicontinuity is proved by assuming
that a global attractor exists. Our result shows that the same is
true for local attractors. We can prove without further assumptions
on the system that such local attractors always exist if the perturbation
is {}``uniformly small''. If $K$ is the global attractor of $\pi_{0}$
and the assumption (h2) of \cite{Caraballo2003} holds then $K_{n}$
must be the global attractor for large $n$.\end{rem*}
\begin{proof}
If $\tilde{N}$ is stable then w.l.o.g. we can replace $\tilde{N}$
by a stable isolating block defined via the function $g^{-}$ (see
below). The previous theorem implies that $N_{n}$ is a stable isolating
neighborhood for $\pi_{n}$ and $n\ge n_{0}$, i.e. $N_{n}=A_{\pi_{n}}^{+}(N_{n})$.
Because $N_{0}^{'}=\mathbb{R}\times N'\subset N_{n}$ we have \[
N_{n}\cap\mathbb{R}\times X\ne\varnothing.\]
Thus there is a sequence $s_{k}\to-\infty$ and $(x_{k})_{k\in\mathbb{N}}$
such that $(s_{k},x_{k})\in N_{n}$. Because $N_{n}$ is positive
$\pi_{n}$-invariant and $N_{0}$ is strongly $\pi_{n}$-semi-admissible
the sequence \[
(s_{k},x_{k})\pi_{n}(-s_{k})\in N_{n}\cap\{0\}\times X\]
has a convergent subsequence and each limit point is in $A_{n}^{-}(N_{n})$
and thus in $A_{\pi_{n}}(N_{n})\subset\operatorname{int}N_{n}^{'}$.
By the same argument we can show that each $t$-slice of $N_{n}$
contains a non-empty compact $t$-slice of the invariant set.

The upper-semicontinuity is a standard result from the index continuation
(see e.g. {\cite[Corollary 4.11]{Carbinatto2002}}).

It remains to show that there is a $\delta>0$ such that for large
$n$\[
U_{\delta}(K_{n})\subset N_{n}.\]
Recall the definition of the function $g^{-}:\tilde{N}\to\mathbb{R}^{\ge0}$
(adjusted to the stable case)\[
g^{-}(x):=\sup\{\alpha(t)F(x\pi t)\,|\,0\le t<\infty\}\]
with \[
F(x):=\min\{1,\operatorname{dist}(x,K)\}\]
and some strictly increasing $C^{\infty}$-diffeomorphism $\alpha:[0,\infty)\to[1,2)$.
$g^{-}$ is continuous on $\tilde{N}$ and strictly decreasing along
$\pi_{0}$ outside of $K$. For sufficiently small $0<\epsilon<\frac{1}{2}$
\[
B_{\epsilon}:=(g^{-})^{-1}([0,\epsilon])\subset\operatorname{int}\tilde{N}\]
defines a stable isolating block and $B_{\delta}\subset B_{\epsilon}$
for $\delta<\epsilon$. Since $g^{-}(x)<1$ for $x\in B$ we have
$\operatorname{dist}(x,K)\le g^{-}(x)$. Because $K$ is compact there
exists some $2\delta<\epsilon$ such that for all $x\in\partial B$
\[
2\delta<d(x,K)\le\epsilon.\]
This implies \[
U_{\delta}(B_{\delta})\subset B_{\epsilon}.\]
The same applies for the suspension $\tilde{g}^{-}:\mathbb{R}\times\tilde{N}\to\mathbb{R}^{\ge0}$.
Our proof of the previous theorem constructs the sets $N_{0}^{'}$
and $N_{0}$ from the function $g^{-}$ (resp. $\tilde{g}^{-}$ in
our adapted version), i.e. \begin{eqnarray*}
N_{0}^{'} & = & (\tilde{g}^{-})^{-1}([0,\epsilon])\\
 & = & \mathbb{R}\times(g^{-})^{-1}([0,\epsilon])\\
 & = & \mathbb{R}\times B_{\epsilon}\end{eqnarray*}
 for some $\epsilon>0$. By the previous argument there is a $\delta>0$
such that for $N_{0}^{''}:=(\tilde{g}^{-})^{-1}([0,\delta])$ we have
\[
U_{\delta}(N_{0}^{''})\subset N_{0}^{'}.\]
 Similar arguments as used in the proof above show that there is a
stable isolating neighborhood $N_{n}^{''}$ of $K_{n}=A_{\pi_{n}}(\mathbb{R}\times\tilde{N})$
with \[
N_{n}^{''}\subset N_{0}^{''}\subset N_{n}^{'}\subset N_{0}^{'}\subset N_{n}\subset N_{0}.\]
Now the inclusion sequence implies \[
U_{\delta}(K_{n})\subset U_{\delta}(N_{n}^{''})\subset N_{n}.\]

\end{proof}

\subsection{Retarded functional differential equations}

Let $C=C([-r,0],\mathbb{R}^{m})$ be the space of continuous functions
equipped with the $\sup$-norm, $r\ge0$ and $\Omega\subset C$ be
an open set. For a continuous map $x:[-r+t,t]\to\mathbb{R}^{m}$ and
$t\in\mathbb{R}$ we write $x_{t}$ as the element in $C$ such that
$x_{t}(\theta)=x(t+\theta)$ for $\theta\in[-r,0].$ 

If $f_{0}:\Omega\to\mathbb{R}^{m}$ is Lipschitz continuous then the
following (autonomous) retarded functional differential equation (RFDE)
\begin{eqnarray*}
\dot{x}(t) & = & f_{0}(x_{t})\end{eqnarray*}
induces a semiflow $\pi_{0}$ such that each bounded and closed $N\subset\Omega$
for which $f_{0}(N)$ is bounded is strongly $\pi_{0}$ admissible
(see \cite{Hale1993,Rybakowski1987}).
\begin{rem*}
An RFDE with $r=0$ is an ODE on $\Omega\subset\mathbb{R}^{m}$.
\end{rem*}
Similarly if $f_{n}:\mathbb{R}\times\Omega\to\mathbb{R}^{m}$ is Lipschitz
continuous then the non-autonomous RFDEs \[
\dot{x}(t)=f_{n}(t,x_{t})\]
induce (skew product) semiflows $\pi_{n}$ on $\mathbb{R}\times\Omega$.
Furthermore, $\pi_{n}$ does not blow up in $\mathbb{R}\times N$
if $N\subset\Omega$ is closed and bounded and $f_{n}(t,N)$ is bounded
uniformly in $t$. One can even show that $\mathbb{R}\times N$ is
strongly $\pi_{n}$-semi-admissible (see proof of {\cite[4.2]{Rybakowski1987}}). 

If we assume that \[
\sup_{t\in\mathbb{R}}\|f_{n}(t,\cdot)-f_{0}\|<\epsilon_{n}\to0\]
then we can easily show that $\pi_{n}\overset{\tiny\mbox{ssing}}{\longrightarrow}\pi_{0}$.
$\mathbb{R}\times N$ is strongly $\{\pi_{n}\}$-semi-singular-admissible.

\subsection{Semilinear parabolic equations}

The idea for semilinear parabolic equations is very similar. Suppose
$A$ is a sectorial operator and $f_{0}:X^{\alpha}\to X$ and $\Omega$
are {}``nice'' then \begin{eqnarray*}
u_{t} & = & Au+f_{0}(u)\\
u|_{\partial\Omega} & = & 0\end{eqnarray*}
induces a local semiflow $\pi_{0}$ on some $X^{\alpha}$. Furthermore,
bounded closed set $N\subset X^{\alpha}$ with $f_{0}(N)$ bounded
are strongly $\pi_{0}$-admissible.

As above the non-autonomous equation \begin{eqnarray*}
u_{t} & = & Au+f_{n}(t,u)\\
u|_{\partial\Omega} & = & 0\end{eqnarray*}
induces a skew product semiflow on $\mathbb{R}\times X^{\alpha}$
and $\mathbb{R}\times N$ is strongly $\pi_{n}$-semi-admissible for
a bounded closed $N\subset X^{\alpha}$ with $f_{n}(t,N)$ is bounded
uniformly in $t$. If, furthermore, \[
\sup_{t\in\mathbb{R}}\|f_{n}(t,\cdot)-f\|<\epsilon_{n}\to0\]
then $\pi_{n}\overset{\tiny\mbox{ssing}}{\longrightarrow}\pi_{0}$
and $\mathbb{R}\times N$ is strongly $\{\pi_{n}\}$-semi-singular-admissible.

\bibliographystyle{amsalpha}
\bibliography{ref}

\end{document}